\documentclass[a4paper,10pt]{amsart}

\usepackage[active]{srcltx}
\usepackage{mathrsfs}
\usepackage[all]{xy}
\usepackage{calrsfs}
\usepackage{times}
\usepackage[scaled=0.92]{helvet}
\usepackage{wrapfig}

\usepackage{amsfonts}
\usepackage{amssymb}
\usepackage[usenames]{color}
\usepackage[dvistyle]{todonotes}

\usepackage[usenames]{color}
\newtheorem{theorem}{Theorem}   

\newtheorem{lemma}[theorem]{Lemma}
\newtheorem{proposition}[theorem]{Proposition}
\newtheorem{corollary}[theorem]{Corollary}
\newtheorem*{problem}{Problem}

\theoremstyle{definition}
\newtheorem{definition}[theorem]{Definition}
\newtheorem{remark}[theorem]{Remark}

\definecolor{MyDarkBlue}{rgb}{0,0.08,0.60}

\newcommand{\Spe}{{\rm Spec}}

\newcommand{\Z}{{\mathbb{Z}}}
\newcommand{\lc}{\left\lceil}
\newcommand{\rc}{\right\rceil}
\newcommand{\lf}{\left\lfloor}
\newcommand{\rf}{\right\rfloor}

\renewcommand{\mod}{{\;\rm mod}}

\title
[Integral represenations, holomorphic differentials]{
Integral representations of cyclic groups acting on relative holomorphic differentials 
of deformations of curves with automorphisms}

\date{\today}

\author{Sotiris Karanikolopoulos }
\address{Freie Universit\"at Berlin \\ Institut f\"ur Mathematik, 
Arnimallee 3,
14195 Berlin, Germany}
\email{skaran@zedat.fu-berlin.de}

\author{Aristides Kontogeorgis}
\address{Department of Mathematics, University of Athens\\
Panepistimioupolis, 15784 Athens, Greece
}
\email{kontogar@math.uoa.gr}

\begin{document}
\bibliographystyle{amsplain}

\begin{abstract}
We study integral representations of holomorphic differentials on the 
Oort-Sekiguci-Suwa component of deformations of curves with cyclic group 
actions. 
\end{abstract} 

\thanks{{\bf keywords:} Automorphisms, Curves, Differentials, Numerical
Semigroups.
 {\bf AMS subject classification} 14H37}

\maketitle 
\section{Introduction}
Let $X$ be a nonsingular projective curve defined over an algebraically closed
field $k$ of 
positive characteristic $p$. If the curve $X$ has genus $g \geq 2$ then it is
known that 
the automorphism group $G$ of $X$ is finite. If $p$ divides $|G|$ then the
automorphism 
group of $X$ behaves in a much more complicated way compared to automorphism
group actions 
in characteristic zero. Wild ramification can appear and the structure of
decomposition 
groups and the different is much more complicated \cite{SeL}. 

The representation theory of groups on holomorphic differentials 
has been a useful tool for studying curves with automorphisms even 
in characteristic zero \cite[chap. V.2]{Farkas-Kra}.
In the positive characteristic case, extra  difficulties  arise in the representation theory of the
 automorphism group
which 
now requires the usage of modular representation theory. 
The classical problem 
of determination of the Galois module structure of (poly)differentials, i.e.
the study of $k[G]$-module structure of $H^0(X,\Omega_X^{\otimes n})$ remains
open in 
positive characteristic and only some special cases are understood
\cite{Madden78},\cite{vm},\cite{csm},\cite{karan}.

Also one can consider the deformation problem of curves with automorphisms: Can
one 
find proper, smooth families $\mathcal{X} \rightarrow \Spe R$ over a local ring $R$, with maximal 
ideal $m_R$, acted on
by $G$, such that the 
special fibre $\mathcal{X} \otimes_{\Spe R} R/m_R$ is the original curve?
There are certain obstructions \cite{BertinCRAS},\cite{ChinburgGuralnickHarbater} 
to the existence of such families
especially if $R$ 
is a mixed characteristic ring. As part of this problem one can consider the 
lifting problem to characteristic zero of a curve with automorphisms. 
We will use the following: 
\begin{lemma} \label{free-lemma}
Let $R$ be a local noetherian integral domain with residue field $k$ and quotient 
field $L$. Every finitely generated $R$-module  $M$ with the additional property 
$\dim_k M\otimes_Rk=\dim_L M\otimes_R L=r$  is free of rank $r$.
\end{lemma}
\begin{proof}
 See \cite[lemma 8.9]{Hartshorne:77}.
\end{proof}
By lemma \ref{free-lemma} the the modules of relative 
polydifferentials
$M_n=H^0(\mathcal{X},\Omega_\mathcal{X}^{\otimes n})$ are  free $R$-modules  that
are 
acted on by $G$.
Aim of this article is to motivate and start the study of the following

\begin{problem}
 Describe the $R[G]$-module structure of $M_n$. 
\end{problem}
This problem is to be studied within the theory of {\em integral
representations}. 
Traditionally the theory of integral representations considers the
$\Z[G]$-module 
structure but essentially the theory of $R[G]$-module structure is similar 
to the $\Z[G]$ theory if  $R$ 
is a principal ideal domain. For cyclic $p$-groups the possible 
$\Z[G]$-modules are classified \cite{IrvingCyc},\cite{Diederichsen}. 
Here the situation is a little bit easier since we will only  consider integral 
domains that contain a $p$-th root of unity.
In this article we will also consider deformations 
over affine schemes where $R$ is not a principal ideal domain.

The study of integral representations in even  more difficult than the theory
of 
modular representations. Actually one of the main ideas of Brauer in studying 
modular representation theory is to lift the representation to characteristic 
zero using complete rings with special fibre $k$.

An open problem is the Oort conjecture that states that the curve can always
be lifted to 
characteristic zero together with the automorphism group, if the automorphism
group 
is cyclic. This conjecture is proved to be true if the order of the cyclic group
is divided 
at most by $p^2$ \cite{MatignonGreen98}. 
We hope that the $R[G]$-module structure for modules of polydifferentials will shed 
some light to the above liftability problem.

The Galois module structure of the special fibre $M_n\otimes_R R/m_R$ remains
open 
but for Artin-Schreier curves is known \cite{Nak:86},\cite{vm},\cite{karan}.
The relative situation in mixed
characteristic
rings is described by the Oort-Sekiguci-Suwa theory \cite{SOS91} and it is also well
understood, at least for cyclic $p$-groups of order $p$. 

Using the Oort-Sekiguci-Suwa theory Bertin-M\'ezard \cite{Be-Me} gave an explicit such model of
equivariant 
and mixed characteristic deformations of an Artin-Schreier cover with only one 
ramification point. In this article we will study explicitly the $R[G]$-module 
structure of $M_1$ using the Bertin-Mezard deformation as a toy model. 
%
%

In order to express our main theorem we introduce first some notation:
\begin{proposition} \label{Vamod}
We extend the binomial coefficient $\binom{i}{j}$ by zero for values  $i<j$.
For every $a \in \mathbb{N}$, $a\leq p$  consider the $a\times a$ matrix $A_a=(a_{ij})$ 
given by 
\[
 a_{ij}=\binom{i-1}{j-1}.
\]
The matrix $(a_{ij})$ is lower triangular and in characteristic $p$ the matrix  $A_a$ 
has order $p$.   Let $G=\langle \sigma \rangle$ be a cyclic 
group of order $p$. The matrix $A_a$ defines an indecomposable $\Z[G]$ module
of dimension $a$ by sending 
\[
 \rho:\sigma^i \mapsto A_a^i.
\]
\end{proposition}
\begin{proof}
Let $k$ be a field of characteristic $p$.
 There is a natural representation of a $p$-group on $k[x]$ by defining $\sigma(x)=x+1$.
Let $k_{a-1}[x]$ be the vector space of polynomials of degree at most $a-1$.
This action with respect to the natural basis $\{1,x,x^2,\ldots,x^{a-1} \}$ has
representation matrix $A_a$. Also the space of invariants 
$k_{a-1}[x]^{\langle \sigma \rangle}$ is one dimensional therefore the 
$G$-module $k_{a-1}[x]$ is indecomposable.  
\end{proof}
For an algebraically closed field $k$ of positive characteristic we consider 
the ring $W(k)[\zeta]$ of Witt vectors with one $p$-root of unity added to it. 
\begin{proposition} \label{Va}
 Let $S$ be an integral domain that is a $W(k)[\zeta]$ algebra. 
Set $\lambda=\zeta-1$,
 For $a_0,a_1 \in \Z$ $a_1<p$  we  consider the $S$-module  
\[V_{a_0,a_1}:=_{S} \left\langle (\lambda X+1)^i: a_0 \leq i \leq a_0+a_1\right\rangle \subset 
S(\lambda X+1).\] 
 Consider the diagonal integral  representation 
of a cyclic group of order $p$ on $V_{a_0,a_1}$ by defining
\begin{equation} \label{1action}
 \sigma (\lambda X+1)^i =\zeta^i (\lambda X+1)^i.
\end{equation}
After an $\mathrm{GL}_a(\mathrm{Quot}(S))$ base change  the representation 
becomes equivalent to 
\begin{equation} 
\label{Marep}
\rho(\sigma)=\mathrm{diag}(\zeta^{a_0},\zeta^{a_0+1},\ldots, \zeta^{a_0+a_1}) A_a.
\end{equation}
The $S[G]$-representation $V_{a_0,a_1}$ is indecomposable. 
\end{proposition}
\begin{proof}
We consider the linear change of coordinates from the $(\lambda X+1)^k$ basis 
to the $X^k$ basis. Observe that this change of basis requires that the elements 
$\lambda$ are invertible. We can work out the conjugation in term of the binomial 
coefficients, but is easier to observe that the action of $\sigma$ on 
polynomials of $X$ is obtained from eq. (\ref{1action}) to be 
$\sigma(X)=\zeta X+1$. 

Let $m$ be the maximal ideal of $W(k)[\zeta]$.
In order to prove that $V_{a_0,a_1}$ is indecomposable we simply take the reduction of 
$S$ modulo  $mS$ and we obtain the indecomposable modular representation of proposition 
\ref{Vamod}. This finishes the proof since a decomposable $S[G]$ integral 
representation should have decomposable reduction as well.
\end{proof}
\begin{remark}
 The isomorphism class of the module $V_{a_0,a_1}$ depends on the equivalence class modulo
 $p$ of $a_0$ and of the length $a_1$. Notice also that the isomorphism class of the reduction 
 of $V_{a_0,a_1}$ modulo the maximal ideal of $S$ depends only on the length $a_1$. 
 Indeed, the $S$-module  isomorphism (but not $G$-module isomorphism
 unless $a_0 \equiv 0 \mod p$)
 \[
  V_{0,a_1} \rightarrow V_{a_0,a_1}
 \]
\[
v \mapsto (\lambda X+1)^{a_0} v 
\]
reduces to the identity modulo the maximal ideal of $S$. This is compatible also with the fact that 
the isomorphism type of a modular representation for the cyclic group depends only on the rank 
of the module. 
\end{remark}

\begin{definition}
 Define by $V_{a}$ the indecomposable integral representation $V_{1,a}$ of proposition \ref{Va}.
It has the matrix representation given in eq. (\ref{Marep}). The module  $V_{a}$ is free of rank 
$a$.
\end{definition}
The main theorem of our article is:
\begin{theorem} \label{main}
Let $\sigma$ be an automorphism of order $p\neq 2$ and conductor $m$ with $m=pq-l$, $1\leq q$, 
$1\leq l \leq p-1$. Let 
\begin{equation} \label{OSS}
R=\left\{
\begin{array}{ll}
  W(k)[\zeta][[x_1,\ldots,x_q]] & \mbox{ if } l=1 \\
 W(k)[\zeta][[x_1,\ldots,x_{q-1}]] & \mbox{ if }
l\neq 1 
\end{array}
\right.
\end{equation}
be the Oort-Sekiguchi-Suwa factor of the versal deformation ring $R_\sigma$. 
The free $R$-module $H^0(\mathcal{X},\Omega_\mathcal{X})$ of relative differentials has the following $R[G]$-structure:
\begin{equation} \label{main-eq}
 H^0(\mathcal{X},\Omega_\mathcal{X})=\bigoplus_{\nu=0}^{p-2}
 V_{\nu}^{\delta_\nu}.
\end{equation}
where 
\[
 \delta_\nu=
\left\{
\begin{array}{ll}
q + \lc \frac{(a+1)l}{p}\rc-\lc \frac{(2+a)l}{p}\rc & \mbox{ if } \nu \leq p-3 \\
q-1 & \mbox{ if } \nu=p-2
\end{array}
\right.
\]
\end{theorem}
We will use the explicit construction of Bertin-M\'ezard \cite{Be-Me} for constructing models 
of the family $\mathcal{X} \rightarrow  \Spe R$. Let $\pi$ be a local uniformizer of $S=W(k)[\zeta]$.
We will employ first the Boseck \cite{boseck} construction for finding a basis for the space of 
differentials for Kummer extensions working with base ring $\mathrm{Quot}(S) \otimes R$. 

Then we will select a basis $B$  so that the $R$-lattice generated by it has full rank and 
gives a well defined reduction modulo the ideal 
 $\pi R$. A detailed analysis for the family of  Artin-Schreier extensions 
given by $\mathcal{X} \times_R \pi R$ proves that the basis $B$ chosen before is indeed a 
basis for $H^0(\mathcal{X},\Omega_{\mathcal{X}})$.
The proof of theorem \ref{main} is given by detailed computations with the basis elements 
chosen. 
As an application we easily  obtain a classical theorem due to Hurwitz on the 
Galois module structure of Kummer extensions.  We also obtain the known modular representation 
structure of the special fiber by reduction of the integral representation.

\section{Explicit deformation theory}
\subsection{Holomorphic differentials of Kummer extensions}
The following theorem will be used for constructing basis on the characteristic zero fibers. 
\begin{theorem}[Boseck]
Let $L$ be a not necessary algebraically closed field of characteristic $p\geq
0$, $(p,n)=1$. 
 Consider the Kummer extension $F$ of $L(x):=F_0$ given by the equation 
$y^n=f(x)$ where $f(x)=a \prod_{i=1}^r p_i(x)^{l_i}$,  the $p_i(x)$ are monic 
irreducible polynomials of degree $d_i$, and the exponents $l_i$ satisfy 
$0 < l_i < n$. The place at infinity of $L(x)$ is assumed to be non ramified in 
$F/F_0$ and this is equivalent to 
$\deg(f)=\sum d_i l_i\equiv 0 \mod n$.
The ramified places in $F/F_0$ correspond to the irreducible polynomials $p_i$ 
and are ramified in extension $F/F_0$ with index $e_i=n/(n,l_i)$. 
We will denote by $g_i$ the number of places of $F$ extending the place $p_i$. 
Set $t:=\deg(f)/n$ and $\lambda_i=e_i l_i/n$. Let $f_i$ be the inertia degree of the
places 
$p_i$ in extension $F/F_0$. For every $\mu=1,\ldots, n-1$ we 
define $m_i^{(\mu)}$, $\rho_i^{(\mu)}$, by the division
\[
 \mu \lambda_i=m_i^{(\mu)} e_i + \rho_i^{(\mu)} \mbox{ where } 
0 \leq \rho_i^{(\mu)} \leq e_i-1.
\]
We set 
\[
 t^{(\mu)}=\frac{1}{n} \sum_{i=1}^r d_i f_i g_i \rho_i^{(\mu)},
\]
and 
\[
 g_\mu(x)=\prod_{i=1}^r p_i(x)^{m_i^{(\mu)}}.
\]
A basis of holomorphic differentials is given by 
\[
 x^\nu g_\mu(x) y^{-\mu} dx,
\]
for $\mu=1,\ldots,n-1$, $0\leq \nu \leq t^{(\mu)}-2$.
\end{theorem}
\begin{proof}
 See \cite[II pp. 48-50]{boseck}
\end{proof}
%

\subsection{The Bertin-M\'ezard model}

Let $k$ be an algebraically closed field of positive characteristic $p>0$.
Consider the Witt ring $W(k)[\zeta]$ extended by $p$-th root of unity, 
and let $L=\mathrm{Quot}(W(k)[\zeta])$. 
Let $\lambda=\zeta-1$. 

We consider the Kummer-extension of rational function field  $L(x)$ defined by
the 
extension 
\[
 (X+\lambda^{-1})^p =x^{-m}+\lambda^{-p}.
\]
Write $m=pq-l$ where $0<l<p-1$.

Set $\lambda X+1=y/x^q$. We have then the model 
\begin{equation} 
\label{def1}
 y^p=(\lambda^p +x^m) x^l.
\end{equation}
Notice that the polynomial $f(x)= (\lambda^p +x^m)x^l$ on the right hand side
has degree
$m+l=qp$ divisible by $p$ so the place at infinity is not ramified. 

More generally we replace $x^q$ by 
\begin{equation} \label{dexq}
 a(x)=x^q + x_1 x^{q-1} + \cdots +x_q,
\end{equation}
where $x_q=0$ if $l\neq 1$.
This gives the  following Kummer extension:
\begin{equation} \label{defff}
 \big( \lambda \xi +a(x) \big)^p=\lambda^p x^l +a(x)^p,
\end{equation}
{  where  $\xi=X a(x)$};
and if we set 
\begin{equation} \label{ydef} 
y=\lambda \xi+a(x)=a(x)(\lambda X+1) 
\end{equation}
 we have:
\begin{equation} \label{def2}
 y^p=\lambda^p x^l +a(x)^p.
\end{equation}
Observe that eq. (\ref{def2})  becomes
eq. (\ref{def1}) if we set $x_1=\cdots=x_q=0$. 

It is known \cite{Be-Me} that  
the  global deformation functor is prorepresentable 
 by a ring $R_\sigma$.  
Bertin-M\'ezard for the case we are studying proved the following 
\begin{theorem}
 Let $\sigma$ be an automorphism of order $p\neq 2$ and conductor $m$, 
with $m=pq-l$, $1\leq q$, $1 \leq l \leq p-1$. The versal 
ring $R_\sigma$ has a formally smooth quotient $R_\sigma \twoheadrightarrow R$  called the 
Oort-Sekiguchi-Suwa (OSS) factor, the ring $R$ is  given  in eq. (\ref{OSS}).
One model of a family over  the OSS factor $R$  is given by setting $\xi=X a(x)$
where $a(x)$ is defined in eq. (\ref{dexq}) and by taking the normalization of the 
fibers given by eq. (\ref{defff}).
\end{theorem}

We will use the Boseck construction for bases for holomorphic 
differentials on the generic fibre. 
Set $\bar{x}=(x_1,\ldots,x_q)$ and consider the polynomial 
\[
f_{\bar{x}}(x)=\lambda^px^l +a(x)^p. 
\]
We consider the decomposition of the polynomial $f_{\bar{x}}(x)$
as a product 
of irreducible polynomials $p_i(x)$  of degree $d_i$
\[
 f_{\bar{x}}(x)=\prod p_i(x)^{l_i}.
\]
Consider the $\bar{x}=(0,\ldots,0)$ case i.e. 
$f_{\bar{x}}(x)=\lambda^p x^l +x^{pq}=x^l(\lambda^p +x^m)$.
In this  case it is clear (compute the discriminant of the 
polynomial $x^m-a$) that $l_i=1$ for irreducible factors of 
$\lambda^p +x^m$ and the only possible factor where 
 $l_1:=l>1$ is $x$.
Since $(l_i,p)=1$ we have that every place $P_i$ corresponding to 
an irreducible factor 
$p_i$ of $f_{\bar{x}}(x)$ ramifies completely in the 
cover $X \rightarrow \mathbb{P}^1_L$.
Therefore, all inertia degrees are $f_i=1$ and the number of 
places $g_i$ above $P_i$ is $g_i=1$.
Notice that the number of places that are ramified is $m+1$.

We can arrive to the same conclusion for the general $f_{\bar{x}}(x)$.
Indeed,
if $l=1$ then the polynomial $f_{\bar{x}}(x)$ is of degree 
$m+1$ and should have $m+1$ distinct roots in an algebraic closure
(otherwise the genus of the curve given in eq. (\ref{def1}) differs 
from the genus of the curve given in eq. (\ref{def2}).
%

For the $l>1$ case notice that $a(x)^p x^{-l}$ is a polynomial and we have
\[
 f_{\bar{x}}(x)=x^l(\lambda^p +a(x)^px^{-l}),
\]
i.e. we have the factor $x$ of multiplicity $l$ and the remaining 
factor  $\lambda^p +a(x)^px^{-l}$ is a polynomial of degree $m=pq-l$
that should have distinct roots, otherwise the genera of the curves defined
in eq.  (\ref{def1}) and (\ref{def2}) are different.

\begin{center}
 \includegraphics[scale=.7]{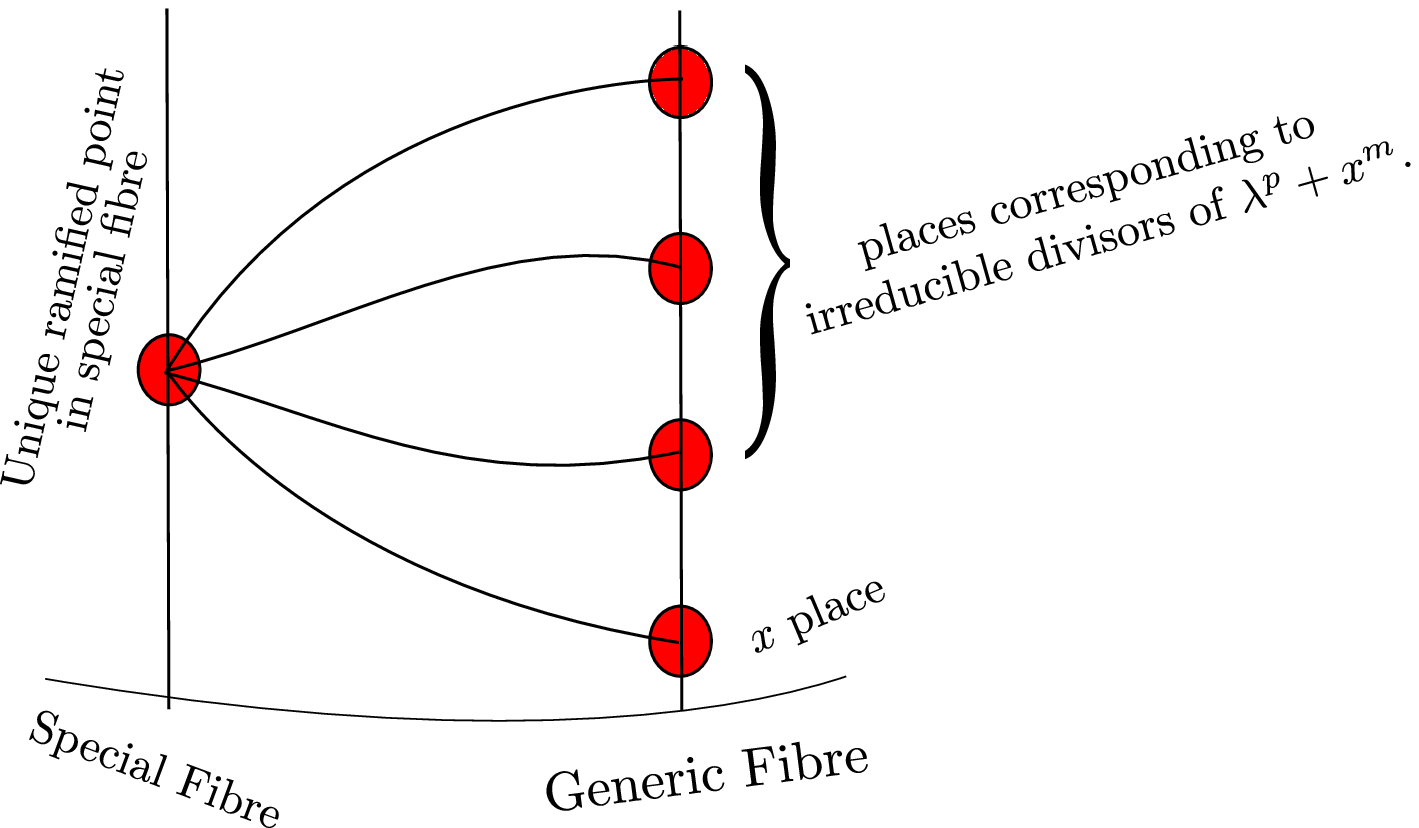}
\end{center}

Following the notation of Boseck we have $\lambda_i=e_i l_i/p=l_i$ and 
we also define for every $\mu=1,2,\ldots,p-1$ the integers 
\[
 \mu \lambda_i=m_i^{(\mu)} e_i + \rho_i^{(\mu)},
\]
i.e.
\[
 m_i^{(\mu)}=\lf \frac{\mu l_i}{p} \rf \mbox{ and }
 \rho_i^{(\mu)}=p\left\langle  \frac{\mu l_i}{p} \right\rangle=\mu l_i-\lf
\frac{\mu l_i}{p} \rf p.
\]
This means that for all { $2\leq i \leq r$} such that $l_i=1$ we have: 
\[
 m_i^{(\mu)}=0 \mbox{ and } \rho_i^{(\mu)}=\mu \mbox{ for all } 1\leq i \leq r,
\]
while the same holds for $m_1,\rho_1$, only if $l=1$ i.e., only if 
$m=pq-1$. 
Therefore, the polynomials 
\[
 g_\mu(x)=\prod_{i=1}^r p_i(x)^{m_i^{(\mu)}}
\]
that are defined in the work of Boseck have as only factor $x^{m_1^{(\mu)}}$.

We set 

\[
 t^{(\mu)}=\frac{1}{p} \sum_{i=1}^r d_i f_i g_i \rho_i^{(\mu)}=
\frac{m\mu+\mu l-\lf \frac{\mu l}{p} \rf p}{p}=\mu q- \lf \frac{\mu l}{p} \rf.
\]

\[
t^{(\mu)}=\frac{1}{p} \sum_{i=1}^r d_i f_i g_i \rho_i^{(\mu)}=
\frac{m\mu+\mu l-\lf \frac{\mu l}{p} \rf p}{p}=\mu q- \lf \frac{\mu l}{p} \rf.
\]
The set of holomorphic differentials are given by 
\begin{equation} \label{BoseckBasis}
 x^{\nu +\lf \frac{l\mu}{p}\rf}y^{-\mu} dx, \mbox{ where } 1\leq \mu \leq p-1, 
0 \leq \nu \leq t^{(\mu)}-2=\mu q- \lf \frac{\mu l}{p} \rf-2.
\end{equation}

Let us write $y=a(x) (\lambda X +1)$, 
set  $N:=\nu +\lf \frac{l\mu}{p}\rf$ and $a:=p-1-\mu$. We have proved the 
\begin{proposition} \label{aa6}
 The set of differentials of the form
 \begin{equation} \label{a-basis}
 x^{N}   a(x)^a \frac{(\lambda X+1)^{a}}{ a(x)^{p-1} 
(\lambda X+ 1)^{p-1}} dx,
\end{equation}
where  
\begin{equation} \label{range-cond}
0 \leq a  < p-1 \mbox{  and  } l-\lc\frac{(1+a)l}{p} \rc \leq N \leq (p-1-a) q-2,
\end{equation} 
forms a basis of holomorphic differentials.
\end{proposition}

This base is not suitable for taking the reduction modulo the maximal ideal of
the ring  $S=W(k)[\zeta]$. We will select a different basis so that it has good reduction.
 
\begin{lemma}\label{agrr}
 For every holomorphic differential  in the basis given in proposition  \ref{aa6}
and for $0 \leq k \leq a$
the  differential 
\begin{equation} \label{want}
  x^{N}   a(x)^a \frac{(\lambda X+1)^{k}}{ a(x)^{p-1} 
 (\lambda X+ 1)^{p-1}} dx
 \end{equation}
is also holomorphic.
\end{lemma}
\begin{proof}
For $k=a$ we have nothing to show. Let us assume that $0\leq k < a$. 
 The differentials of the form 
\begin{equation} \label{have}
  x^{N'}   a(x)^{k} \frac{(\lambda X+1)^{k}}{ a(x)^{p-1} 
(\lambda X+ 1)^{p-1}} dx
\end{equation}
for $ l-\lc\frac{(1+k)l}{p} \rc \leq N' \leq (p-1-k) q-2$ are holomorphic. We will show that 
every differential given in eq. (\ref{want}) is written as a linear combination 
of holomorphic differentials given in (\ref{have}).

Indeed we have to show that we can select $\lambda_{N'}$ such that 
\[
 \sum_{N'=l-\lc\frac{(1+k)l}{p}\rc }^{(p-1-k) q-2} \lambda_{N'} 
 x^{N'}   a(x)^{k} 
=
x^{N}   a(x)^a, 
\]
 or equivalently that 
\begin{equation} \label{le-a}
 \sum_{N'=l-\lc\frac{(1+k)l}{p}\rc }^{(p-1-k) q-2} \lambda_{N'} 
 x^{N'}  
=
x^{N}   a(x)^{a-k}. 
\end{equation}
We now observe that  $a(x)^{a-k}$ is a polynomial of degree $q(a-k)$ and the possible
 values of $N$ are in the range given by (\ref{range-cond}). This means that 
the in  right hand side of eq. (\ref{le-a}) appear monomials  of degree
  at most $N_1=N+q(a-k)$ where $N_1$ satisfies
  \[
     N_1 \leq (p-1-a) q-2+q(a-k)=(p-1-k) q-2.
  \]
and at least  
\[
 N_2:=\left\{ 
\begin{array}{ll}
 N  &\mbox{ if } l=1\\
N+a-k & \mbox{ if } l>1
\end{array}
\right.
\]
The distinction in the two cases appears since in the $l>1$ case the polynomial $a(x)$ has 
zero constant term.

For $l=1$ it is enough to prove that  
\[
 l-\lc\frac{(1+k)}{p}\rc \leq  l-\lc\frac{(1+a)}{p} \rc
\]
which is immediate since $1\leq 1+k \leq 1+a \leq p-1$. 

For the $l>1$ we have to prove that 
\begin{equation} \label{prol1}
  \lc\frac{(1+a)l}{p} \rc-\lc\frac{(1+k)l}{p} \rc \leq (a-k).
\end{equation}
Write $a=k+t$ for some $t>0$. 
Then we have
\[
  (1+a)l=\pi_ap+v_a,\;\; (1+k)l=\pi_kp+v_k \mbox{ where } 0\leq v_a,v_k < p, 
\]
so 
\[
 (1+a)l=\pi_kp+v_k +tl
\]
and 
\[
 v_k+tl=p \lf \frac{v_k+tl}{p}\rf+v \mbox{ for some } 0\leq v <p.
\]
This way we see that 
\[
 \pi_a=\pi_k+  \lf \frac{v_k+tl}{p}\rf, 
\]
%
%
%
but  
\[
 \lf \frac{v_k+tl}{p}\rf \leq \lf \frac{v_k+tp}{p}\rf =t+ \lf \frac{v_k}{p}\rf=t,  
\]
therefore (\ref{prol1}) holds.
%
\end{proof}

\begin{corollary} \label{le-a-g}
 For every $p(x) \in R[x]$ so that the  differential 
$p(x) a(x)^a \frac{(\lambda X+1)^a}{a(x)^{p-1} (\lambda X+1)^{p-1}} dx$ is 
holomorphic,
the differentials
$p(x) a(x)^a \frac{(\lambda X+1)^k}{a(x)^{p-1} (\lambda X+1)^{p-1}} dx$
are also holomorphic.
\end{corollary}

For $0 \leq a \leq p-2$ we consider the set of admissible $N$ i.e. $N$ that satisfy 
the inequalities $ \underline{N_a} \leq N \leq \overline{N_a}$, where 
$  \underline{N_a}= l-\lc\frac{(1+a)l}{p} \rc$ and $\overline{N_a}=(p-1-a) q-2$.
For example $\underline{N_{p-2}}= 0$ and $\overline{N_{p-2}}=q-2$.

We consider the space of holomorphic differentials as a graded space with grading given 
by the exponent of $(\lambda X+1)^a$, i.e.
\[
 \Omega_X =\bigoplus_{0 \leq a \leq p-2} \Omega_X^a.
\]
Notice that every space $\Omega_X^a$ has dimension 
\begin{equation} \label{dimO}
 \dim \Omega_X^a=\overline{N_a} -\underline{N_a}+1=(p-1-a) q-2-l+\lc\frac{(1+a)l}{p}\rc+1.
\end{equation}
 Corollary \ref{le-a-g} implies that  for every $0 \leq k  \leq a \leq p-2$ 
there are linear injections 
\[
L_{a,k}:\Omega_X^{a} \rightarrow \Omega_X^{k}, 
\]
so that 
\[
 L_{a,k}\left(p(x) a(x)^a \frac{(\lambda X+1)^a}{a(x)^{p-1} (\lambda X+1)^{p-1}} dx \right)=
p(x) a(x)^a \frac{(\lambda X+1)^k}{a(x)^{p-1} (\lambda X+1)^{p-1}} dx.
\]
Define 
\[
C_{p-2}:=\Omega_X^{p-2}.
\]
We define $C_{p-3}$ so that  $\Omega^{p-3}_X$  is a direct sum in the category 
of vector spaces:
\[
 \Omega^{p-3}_X=L_{p-2,p-3}(C_{p-2}) \oplus C_{p-3}.
\]
We proceed inductively. We define $C_{p-4}$ so that 
\[
 \Omega^{p-4}_X=L_{p-2,p-4}(C_{p-2}) \oplus L_{p-3,p-4}(C_{p-3}) \oplus C_{p-4}.
\]
 More generally:
\begin{equation} \label{L-sum}
 \Omega_X^{p-k}=\bigoplus_{\nu=p-k+1}^{p-2} L_{\nu,p-k}(C_\nu) \oplus C_{p-k}=
\bigoplus_{\nu=p-k}^{p-2} L_{\nu,p-k}(C_\nu).
\end{equation}
After setting $a=p-k$ we have 
\[
 \Omega_X^{a}=
\bigoplus_{\nu=a}^{p-2} L_{\nu,a}(C_\nu).
\]
We now write
\[
 \Omega_X=\bigoplus_{a=0}^{p-2}  \bigoplus_{\nu=a}^{p-2}
L_{\nu,a} (C_\nu) 
=
\bigoplus_{\nu=0}^{p-2}  \bigoplus_{a=0}^{\nu}
L_{\nu,a} (C_\nu)
.
\]
Fix a basis 
\[
 \left\{f_1 \frac{(\lambda X+1)^\nu}{a(x)^{p-1}(\lambda X+1)^{p-1}}dx,\ldots,
f_{\dim C_\nu} \frac{(\lambda X+1)^\nu}{a(x)^{p-1}(\lambda X+1)^{p-1}}dx \right\}
\]
of $C_\nu$, where $f_i$ are polynomials in $R[x]$. 
Then the set 
\[
 \left\{f_\mu \frac{(\lambda X+1)^a}{a(x)^{p-1}(\lambda X+1)^{p-1}}dx:
 1\leq \mu \leq \dim C_\nu, 0\leq a \leq \nu \right\}
\]
is a basis of the space $\bigoplus_{a=0}^{\nu}
L_{\nu,a} (C_\nu)$.
Furthermore,
proposition \ref{Va} implies that 
\[
 \bigoplus_{a=0}^{\nu}
L_{\nu,a} (C_\nu)=V_{\nu}^{\dim C_\nu}.
\]
\begin{corollary}\label{final-basis}
A  natural basis for $\bigoplus_{a=0}^{\nu}
L_{\nu,a} (C_\nu)$ with respect to the reduction is 
\begin{equation}
 \left\{f_\mu \frac{ X^a}{a(x)^{p-1}(\lambda X+1)^{p}}dx:
 1\leq \mu \leq \dim C_\nu, 1\leq a \leq \nu+1 \right\}
\end{equation}
\end{corollary}
Using the decomposition in eq. (\ref{L-sum}) and eq. (\ref{dimO}) we obtain that 
\[
\dim C_{p-2}=\dim \Omega_X^{p-2}=q-1.
\]
\begin{eqnarray*}
 \dim C_{p-k}&=& \dim \Omega_X^{p-k} -\dim \Omega_X^{p-k+1} \mbox{ for } k\geq 3 \\
\end{eqnarray*}
So
\[
 \dim C_a=
\left\{
\begin{array}{ll}
\dim \Omega_X^a -\Omega_X^{a+1}=q + \lc \frac{(a+1)l}{p}\rc-\lc \frac{(2+a)l}{p}\rc & \mbox{ if } a \leq p-3 \\
q-1 & \mbox{ if } a=p-2
\end{array}
\right.
\]
The proof of our main theorem \ref{main} will be complete if we show that 
the differentials chosen in corollary \ref{le-a-g} are $\mathcal{X}$ holomorphic. 
For this we will study the  characteristic $p$-fibers of our family $\mathcal{X}\rightarrow \Spe R$. 
\section{On the finite characteristic fibres}
In this section we consider the reduction of the model given in eq. \ref{def2}
\[
 \big( \lambda \xi +a(x) \big)^p=\lambda^p x^l +a(x)^p,
\]
modulo the ideal generated by $\pi$. Using that for $\lambda=\zeta-1$, $p \lambda^{-j} \equiv 0 \mod \pi$ for 
$0\leq j <p-1$ and  that $p \lambda^{-(p-1)}\equiv -1 \mod \pi$ Bertin-M\'ezard in \cite[sec. 4.3]{Be-Me}
arrived to the equation:
\begin{equation} \label{AS-red} 
 X^p-X=\frac{x^l}{a(x)^p}, \mbox{ where } X=\frac{\xi}{a(x)}.
\end{equation}
This equation is  Artin-Schreier but not in normal form since the right hand side has poles of orders 
divisible by $p$.
Write (keep in mind that  $l_1=l$)
\[
a(x)=x^{l_1} \prod_{i=2}^r p_i(x)^{l_i}, 
\]
where $p_i(x)$ are irreducible polynomials in $S[[x_1,\ldots,x_q]][x]$.
The valuations of the denominator in (\ref{AS-red}) are not prime to $p$. 
Bertin-M\'ezard proved that  the normalization of $R[[x]]$ in the Galois extension 
of the generic fibre is the ring $R[[\eta]]$, where $\eta^l=\xi$. 
The group action of the generator $\sigma$ of $G$ is then given by 
\[
 \sigma(\xi)=\xi +a(x) \Rightarrow \sigma(\eta)=\eta \left(1+\frac{a(x)}{\eta^l} \right)^{1/l}
\]
so 
\begin{eqnarray} 
 \sigma(\eta)-\eta &=&   \eta \sum_{\nu=1}^\infty \binom{\frac{1}{l}}{\nu} 
\left(\frac{a(x)}{\eta^l}\right)^\nu \nonumber \\
& =& \left(\frac{a(x)}{\eta^{l-1}}\right) \left( \sum_{\nu=1}^\infty \binom{\frac{1}{l}}{\nu} 
\left(\frac{a(x)}{\eta^l}\right)^{\nu-1} \right). \label{11-aa}
\end{eqnarray}
Notice that for $l>1$ the polynomial $a(x)$ has at least one root, and since $x=\eta^p u(\eta)$, 
where $u(\eta)$ is a unit in $R[[\eta]]$ we have that $a(x)/\eta^{l-1}$
has reduced order $pq-l+1=m+1$ and by Weierstrass preparation theorem can be written as
\[
 \left(\frac{a(x)}{\eta^{l-1}} \right)=\left(\eta^{pq-l+1} +a_{pq-l}(x_1,\ldots,x_q) \eta^{pq-l} + \cdots 
+a_0(x_1,\ldots,x_q) \right) U(\eta), 
\]
where $a_i(0,\ldots,0)=0$. According to \cite[sec. 2.1]{MatignonGreen98} the ramification locus 
corresponds to the irreducible factors  of distinguished Weierstrass polynomial that differs to $a(x)$ only by 
a unit. Therefore only the places $p_i(x)$ corresponding to factors of $a(x)$ are ramified. If $l_i$ 
is the multiplicity of the polynomial $p_i(x)$ in the decomposition of $a(x)$ then the conductor 
is given by 
\[
 m_i=\left\{
\begin{array}{ll}
 p l_i-1 & \mbox{if } i\neq 1 \\
 p l_1-l_1 & \mbox{if } i=1.
\end{array}
\right.
\]
\begin{remark}
 Notice that since $\sum l_i=q$, the contributions to the different 
\[
 \sum (p-1)(m_i+1)= (p-1)(pq-l+1)=(p-1)(m+1)
\]
are as expected \cite[sec. 3.4]{MatignonGreen98},\cite[sec. 5]{KatoDuke87}.
\end{remark}
Let $P_i$ be the unique places above the polynomials $p_i(x)$ and $P_0$ above $x$.
We  compute the divisors:
\[
 \mathrm{div}(X)=\mathrm{div}_0(X)- \sum_{i=2}^r l_ip  P_i- (l_1p-l) P_1.
\]
\begin{remark}
 Notice that in the case of normalized Artin-Schreier curves the coefficients in front of poles of 
the generating functions is just the conductor of the corresponding place see \cite[eq. 26]{boseck}.
This can not be true in our case  since the conductors are not divisible by $p$. 
\end{remark}
We compute 
\[
 \mathrm{div}(dx)=\sum_{i=2}^r (p-1)pl_iP_i+ (p-1)(l_1p-l+1)P_1 -2 \mathrm{Con}(P_\infty),
\]
where $\mathrm{Con}(P_\infty)$ is the sum of places extending $P_\infty$ in the Artin-Schreier extension, 
and 
\[
 \mathrm{div}(X^\mu dx)=\sum_{i=2}^r
(p-1-\mu)pl_i
 P_i
+\big(
(p-1-\mu)(l_1p-l)+(p-1)
\big)P_1
-2 \mathrm{Con}(P_\infty) +\mu \mathrm{div}_0(X).
\]
Following Boseck we now set
\[
 m_i^{(\mu)}=(p-1-\mu)l_i \mbox{ and } 
m_1^{(\mu)}=\lf \frac{(p-1-\mu)(l_1p-l)+(p-1)}{p} \rf
\]
Then we form the polynomials
\begin{eqnarray*}
 g_\mu(x) & = &\prod_{i=1}^r {p_i(x)}^{m_i^{(\mu)}}= 
\left(\frac{a(x)}{x^{l_1}}\right)^{p-1-\mu} x^{\lf \frac{(p-1-\mu)(l_1p-l)+(p-1)}{p}\rf}\\
&=& a(x)^{p-1-\mu} x^{\lf \frac{-l(p-1-\mu)+(p-1)}{p}\rf}.
\end{eqnarray*}
Therefore 
\[
 g_{\mu}(x)^{-1} X^\mu dx 
\]
is a holomorphic differential if 
\[
 t^{(\mu)}=\sum_{i=1}^r d_i m^{(\mu)}_i=q(p-1-\mu)-l +\lf \frac{l(1+\mu)+p-1}{p}\rf \geq 2.
\]
Notice that for $\mu=p-1$ the above formula gives $t^{(p-1)}=0$ so $\mu=p-1$ is not 
permitted and $0\leq \mu \leq p-2$. 
By computation we see
\[
 \left\{x^{\nu}g_\mu(x)^{-1} X^\mu dx, \;\; \mu=0,\ldots,p-1, t^{(\mu)} \geq 2,\nu=0,\ldots,t^{(\mu)}-2
\right\}
\]
is  a basis of holomorphic differentials. 
We set $N:=\nu-{\lf \frac{-l(p-1-\mu)+(p-1)}{p}\rf}$ 
and we observe that 
\begin{equation} \label{ine11}  
l-\lf \frac{l(1+\mu)+(p-1)}{p}\rf \leq  N\leq q (p-1-\mu)-2.
\end{equation}
Since for every integer $a$ we have $\lf \frac{a+p-1}{p}\rf =\lc \frac{a}{p}\rc$ we 
see that the inequality (\ref{ine11})  is equivalent to (\ref{range-cond}) and the set    
\[
 \left\{x^{N} a(x)^\mu \frac{X^\mu}{a(x)^{p-1}} dx
\right\}
\]
with $N$ satisfying (\ref{ine11}) and $0\leq \mu \leq p-2$ 
is the reduction of the set given in (\ref{final-basis}).
This proves that the union of the  bases given  (\ref{final-basis}) is  not just 
a basis  of holomorphic differentials on the characteristic zero fibers but 
forms a free basis of the module of relative differentials $H^0(\Omega_{\mathcal{X}},\mathcal{X})$.


%
%

\subsection{On a theorem of Hurwitz}
The following is a classical result due to Hurwitz \cite{Hu:1893} (see also \cite[Theorem 3.5 p. 600]{MorrisonPinkham})
 that characterizes the dimension of the 
$\zeta^i$ eigenvalues of the generator of a  $p$-cyclic group on the space of holomorphic differentials. 
\begin{theorem}[Hurwitz]\label{cyclic prime to p}
Let $F/E$ be a cyclic Galois extension of function fields with Galois group
$C_n$, $(p,n)=1$. 
This is given in Kummer form $y^n=u$, where $u\in E$ and $u^l \not\in E$ for
every $l\mid n$.
For every place $P$ of $E$ we set  $\Phi(i)=v_{P_i}(y) =\frac{e_i
v_{\bar{P}_i}(u)}{n}$. Set 
\[
 \Gamma_k:=\sum_{i=1}^r\left\langle \frac{k\Phi(i)}{e_i}\right\rangle.
\]
For $k=0,\ldots, n -1$, we have $n$ distinct  irreducible representations of
degree $1$. 
The $k$th representation occurs $d_{k}=\Gamma_{n-k} -1+g_E$ times in the
representation of $G$ 
in $\Omega_F$, when $k\neq 0$ and $g_E$ times when $k=0$.
\end{theorem}
In our case we have $\Phi(i)=1$ for all places that correspond to divisors of
$\lambda^p+x^m$ and 
$\Phi(r)=l$ for the last place.
We compute
\begin{eqnarray*}
 \Gamma_k &=&\sum_{i=1}^r \left\langle \frac{k\Phi(i)}{p}\right\rangle \\
 &=& m\left( \frac{k}{p}-  \lf \frac{k}{p} \rf\right) + \frac{kl}{p}-  \lf
\frac{kl}{p} \rf\\
 &=&  k\frac{m+l}{p}- \lf \frac{kl}{p} \rf \\
 &=& kq- \lf \frac{kl}{p} \rf.
\end{eqnarray*}
The dimension of the eigenspace of the eigenvalue $\zeta^\mu$ is
$\Gamma_{p-\mu}-1$.
We compute:
\begin{equation} \label{hur}
{\Gamma_{p-\mu}-1}  =  (p-\mu)q -  \lf \frac{(p-\mu)l}{p} \rf-1.
\end{equation}
We would like to obtain the result of Hurwitz using our result on integral representation.
Observe that the action of $\sigma$ on $V_{n}$ has eigenvalues
$\zeta,\zeta^2,\ldots,\zeta^{n+1}$ thus there is contribution from $V_{n}$ to the 
eigenspace of the eigenvalue $\zeta^\mu$  only if $n+1 \geq \mu$.
 Therefore only the summands for which $n \geq \mu-1$ contribute $\dim C_n$ to the 
eigenspace. The dimension of the desired eigenspace is:
\begin{eqnarray*}
 \sum_{n=\mu-1}^{p-2} \dim C_n & = & \sum_{n=\mu-1} ^{p-3} \dim C_n + \dim C_{p-2}\\
&=&q(p-\mu-1)+\lc \frac{\mu l}{p} \rc - \lc \frac{(p-1)l}{p}\rc +q-1\\
&=& q(p-\mu)+\lc \frac{\mu l}{p}\rc -l -1.
\end{eqnarray*}
Of course this result coincides with the result given by eq. (\ref{hur}). 

 \def\cprime{$'$}
\providecommand{\bysame}{\leavevmode\hbox to3em{\hrulefill}\thinspace}
\providecommand{\MR}{\relax\ifhmode\unskip\space\fi MR }
\providecommand{\MRhref}[2]{%
  \href{http://www.ams.org/mathscinet-getitem?mr=#1}{#2}
}
\providecommand{\href}[2]{#2}


\end{document}